\documentclass{amsart}

\usepackage{amssymb}
\usepackage{amsmath}
\usepackage{amsthm}
\usepackage{amsbsy}
\usepackage{amscd}
\usepackage{bm}
\usepackage[dvips]{graphicx}
\usepackage{color}

\theoremstyle{plain}
\newtheorem{theorem}{Theorem}[section]
\newtheorem{lemma}[theorem]{Lemma}
\newtheorem{corollary}[theorem]{Corollary}
\newtheorem{proposition}[theorem]{Proposition}

\theoremstyle{definition}

\theoremstyle{remark}
\newtheorem{remark}[theorem]{Remark}
\newtheorem*{acknowledgements}{Acknowledgements}

\newcommand{\Z}{\mathbb{Z}}

\newcommand{\CC}{\mathcal{C}}

\newcommand{\del}{\partial}
\newcommand{\bdel}{\bar{\partial}}

\begin{document}

\title[Chain complex for normal surfaces]{A chain complex and
Quadrilaterals for normal surfaces}

\author{Siddhartha Gadgil}

\address{   Department of Mathematics\\
        Indian Institute of Science\\
        Bangalore 560012, India}
\email{gadgil@math.iisc.ernet.in}

\author{Tejas Kalelkar}
\address{  Stat-math Unit\\
          Indian Statistical Institute\\
       Bangalore 560059, India}
\email{tejas@isibang.ac.in}

\date{\today}


\begin{abstract}
We interpret a normal surface in a (singular) three-manifold in terms
of the homology of a chain complex. This allows us to study the
relation between normal surfaces and their quadrilateral
co-ordinates. Specifically, we give a proof of an (unpublished) observation
independently given by Casson and Rubinstein saying that quadrilaterals determine a normal surface up to
vertex linking spheres. We also characterise the quadrilateral
coordinates that correspond to a normal surface in a (possibly ideal)
triangulation.
\end{abstract}

\maketitle

\section{Introduction}

A \emph{normal arc} in a triangle is an arc separating a vertex from
the opposite edge. Normal arcs in a triangle, up to isotopy through
normal arcs, are in bijective correspondence with vertices.  A
\emph{normal disc} in a tetrahedron is either a triangle separating a
vertex from the opposite face or a quadrilateral separating a pair of
edges. Normal triangles in a tetrahedron are determined, up to isotopy
through normal discs, by the vertex they separate from its opposite
face. Normal quadrilaterals are determined up to isotopy through
normal discs by the pair of edges they separate. Thus, normal discs in
a tetrahedron are of seven \emph{types}, i.e., isotopy classes.

Given a triangulated $3$-manifold $M$, a \emph{normal surface}
$S\subset M$ is a properly embedded surface in $M$ that intersects each tetrahedron $\Delta$ of the triangulation in a disjoint union
of normal discs. Such a normal surface is determined, up to isotopy
through normal surfaces, by the number of normal discs of each type,
i.e., by $7t$ integers called the \emph{normal coordinates}, where $t$
is the number of tetrahedra in the triangulation.

For $S$ to be a surface, these coordinates satisfy \emph{matching
equations}. Namely, if $F$ is a face contained in two tetrahedra
$\Delta_+$ and $\Delta_-$ and $D$ is a normal disc in one of the
tetrahedra $\Delta_\pm$, then $D \cap F$ is a normal arc. Thus, the normal
discs of $S \cap \Delta_\pm$ give a collection of normal arcs in $F$. As
this coincides with $S \cap F$, we see that the number of arcs in $F$
of each type obtained from the normal discs in the tetrahedra $\Delta_+$
and $\Delta_-$ must coincide.

There are two further conditions for a collection of normal
coordinates to represent an embedded normal surface. Firstly, all the
co-ordinates should be non-negative. Secondly, embeddable surfaces
cannot have quadrilaterals of two different types in a tetrahedron. We
call normal coordinates satisfying this condition on quadrilaterals as
admissible.\\

Casson and Rubinstein independently observed that a normal surface is essentially determined by its
normal quadrilaterals. More precisely, for each vertex $v$, we
consider the normal triangles in tetrahedra containing $v$ that
separate $v$ from the opposite face. The union of these form the
\emph{vertex linking sphere} $S(v)$. These clearly have no
quadrilaterals. Their (unpublished) observation was that normal
surfaces are determined up to vertex linking spheres by quadrilateral
coordinates. This allows a considerable increase in efficiency of
algorithms based on normal surfaces.

The purpose of this note is to clarify this observation, as well as
the complementary question of when a given set of quadrilateral
coordinates corresponds to a normal surface, by interpreting normal
surfaces in terms of the homology of a chain complex associated to
$M$. Our methods also allow us to address the analogous questions for
\emph{ideal triangulations}. A criterion for quadrilateral coordinates
determining a normal surface and a proof of Casson-Rubinstein's observation was
earlier given by Tollefson\cite{Tol} for compact manifolds, using geometric constructions. Tillmann \cite{Til} proves a similar result for ideal triangulations in the context of spun-normal surfaces. Spun-normal surfaces, introduced by Thurston, are the analogue of normal surfaces in ideal triangulations.

As we wish to consider ideal triangulations, we consider a context
more general than triangulated $3$-manifolds. Namely, let $M$ be an orientable
three-dimensional simplicial complex that is a manifold away from
vertices, and so that the link of each vertex $v$ is a closed,
connected, orientable surface (not necessarily a sphere). We can define
normal surfaces in this situation exactly as in the case of
$3$-manifolds.  For a detailed treatment of spun-normal surfaces we refer to \cite{Til}.

Henceforth, we assume $M$ is as above. We can associate to a vertex $v$
the vertex linking normal surface $S(v)$, which is a closed orientable
surface (but not in general a sphere). The space $\hat{M}$ obtained
from $M$ by deleting those vertices $v$ for which $S(v)$ is not a
sphere is a (non-compact in general) $3$-manifold with an ideal triangulation.

\section{The chain complex}
In this section, we associate a chain complex $(\CC,\del_*)$ to $M$ such that normal surfaces are in bijection with cycles of $\CC_2$.

Fix an orientation of $M$. For each vertex $v$, assume that $S(v)$ is oriented so that its co-orientation at each point is along a vector pointing away from $v$ (we make this precise later). As $S(v)$ is a union of normal triangles (linking $v$), we get a triangulation of $S(v)$. Let $(C_*(v),\del_*(v))$ be the simplicial chain complex associated to this triangulation. Then, we shall show that $C_2(v)$ embeds in $\CC_2$, $C_1(v)$ embeds in $\CC_1$ and the restriction of the boundary map $\del_2:\CC_2\to\CC_1$ to $C_2(v)$ agrees with $\del_2(v)$.

\subsection{The chain complex $(\CC_*, \del_*)$}

A normal arc is uniquely determined up to normal isotopy by the face in which it lies and the vertex that it links. Let $v$ be a vertex of a face $F$. We denote by $\alpha(F,v)$ the normal arc that lies in $F$ and links $v$.

We give an arbitrary orientation to the edges of the triangulation of $M$ and let $e(F, v)$ denote the edge in $F$ opposite to $v$. We orient the normal arc $\alpha(F, v)$ so that it is in the same direction as $e(F, v)$. Let $\CC_1$  be the free abelian group generated by these oriented normal arcs up to normal isotopy.

Let $\CC^t_2$ be the free abelian group generated by normal triangles (up to normal isotopy) and $\CC^q_2$ be the free abelian group generated by normal quadrilaterals (up to normal isotopy). Define $\CC_2 = \CC^t_2 \oplus \CC^q_2$ to be the free abelian group generated by normal disks (up to normal isotopy). Finally, for all $k < 1$ and $k > 2$, let $\CC_k$ be zero.

Next we define the boundary maps of $(\CC_*, \del_*)$. Take $\del_k$ to be zero for all $k \neq 2$. To define the boundary map $\del_2$, we proceed as follows.

\begin{figure}[t]
\centering
\includegraphics[scale=0.5]{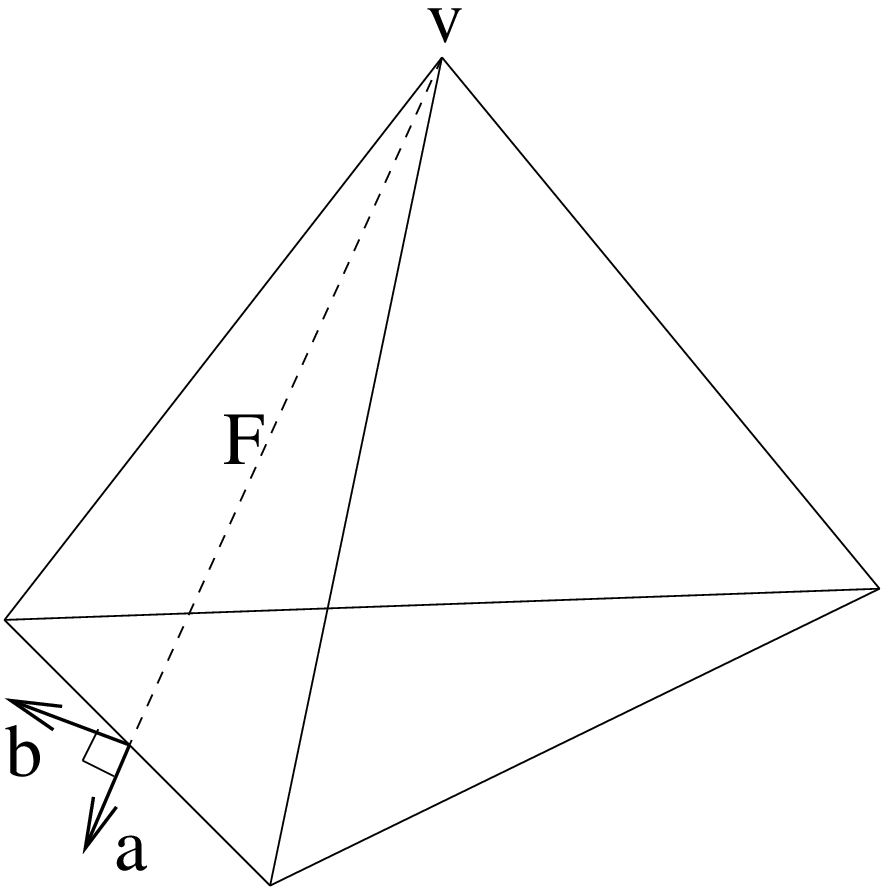}
\hspace{1cm}
\includegraphics[scale=0.5]{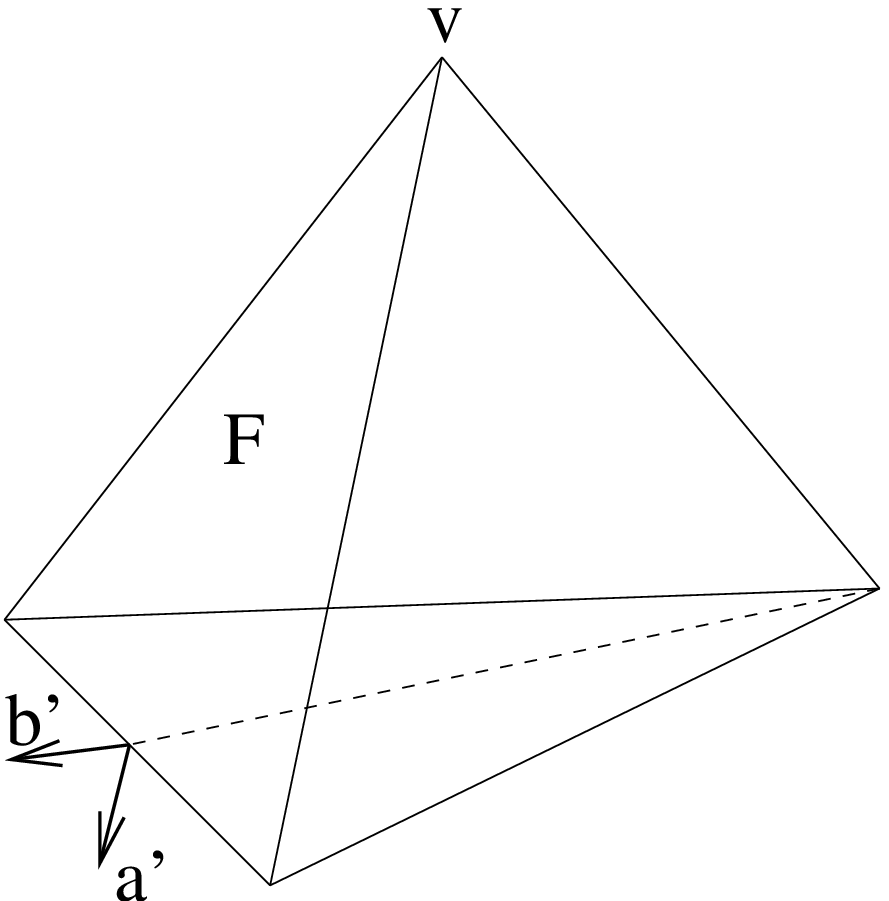}
\caption{The vectors $a$, $b$, $a'$ and $b'$}
\label{fig:Delta}
\end{figure}

Let $v$ be a vertex of a face $F$ of a tetrahedron $\Delta$ (which we identify with a unit simplex in Euclidean space respecting orientations). Let $e=e(F, v)$ denote the edge in $F$ opposite to $v$, and let $m_e(F, v)$ denote its midpoint. Let $a=a(\Delta, F, v)$ denote the unit vector based at $m_e(F, v)$ that is contained in the plane containing $F$ which is normal to $e(F,v)$  and points out of $F$. Let $b=b(\Delta, F, v)$ denote the unit vector based at $m_e(F, v)$, perpendicular to $F$ which points out of $\Delta$ (see figure~\ref{fig:Delta}). Then if $e(F,v)$ is regarded as a unit vector based at $m_e(F,v)$, there is a unique sign $\varepsilon=\varepsilon(\Delta,F,v)=\pm 1$ such that $\langle a,b, \varepsilon e\rangle$ is a  \emph{positively oriented} orthonormal basis. We define $\varepsilon(\Delta,F,v)$ to be this sign.

Observe that if $\Delta_i$, $i=1,2$, are the tetrahedra containing a face $F$, then $\varepsilon(\Delta_1,F,v)=-\varepsilon(\Delta_2,F,v)$ as we have the relations $a(\Delta_1, F,v)= a(\Delta_2, F,v)$ and $b(\Delta_1, F,v)= - b(\Delta_2,F,v)$. We denote by $\Delta_+(F)$ the tetrahedron containing $F$ such that $\varepsilon(\Delta,F,v)=1$, with the other tetrahedron containing $F$ denoted $\Delta_-(F)$.

Given a normal disk $D$ in $\Delta$, suppose that $\del D$ is the union of normal arcs $\{\alpha(F, v)\}_{(F, v) \in A}$. Recall that $\alpha(F,v)$ is oriented in the direction of $e(F, v)$. The boundary map $\del_2 (D)$ is defined to be $$\sum_{(F, v) \in A} \varepsilon(\Delta,F,v) \alpha (F,v)$$
This extends uniquely to a homomorphism $\del_2:\CC_2\to \CC_1$.

\subsection{Normal surfaces and the chain complex}
We can interpret normal surfaces in terms of the chain complex
$(\CC_*,\del_*)$ as follows.

\begin{lemma}\label{match}
There is a bijective correspondence between normal coordinates and
$2$-chains of the chain complex. Further, normal coordinates
corresponding to a $2$-chain $\xi$ satisfy the matching equations if
and only if $\del_2 \xi=0$.
\end{lemma}
\begin{proof}
The first statement follows as $\CC_2$ is the free abelian group
generated by normal isotopy classes of normal discs.

Let $\xi=\sum_j c_jD_j$ be a $2$-chain. Let $F$ be the common face of tetrahedra $\Delta_+(F)$ and $\Delta_-(F)$ and let $\alpha(F,v)$ be a normal arc. By construction, the coefficient of $\alpha(F,v)$ is the difference
$$\sum_{D_i\subset \Delta_+(F)} c_i-\sum_{D_j\subset \Delta_-(F)} c_j$$

Hence the boundary of a 2-chain $\xi$ is zero if and only if for each normal arc $\alpha(F,v)$ in the face $F = \Delta_+ \cap \Delta_-$,
the number $\sum_{D_i\subset \Delta_+(F)} c_i$ of normal disks (counted with sign) of $\xi$ in $\Delta_+$ that have $\alpha$ in their boundary equals the number $\sum_{D_j\subset \Delta_-(F)} c_j$ of normal disks of $\xi$ in $\Delta_-$ having $\alpha$ in their boundary. This is precisely when $\xi$ is a solution of the matching equations. Therefore, $\del_2\xi = 0$ if and only if its normal coordinates satisfy the matching equations.
\end{proof}

Thus, as there are no three-chains, normal surfaces are in bijective
correspondence with the homology $H_2(\CC)$.

\subsection{The inclusion of chain complexes}

For a vertex $v$, the $1$-chains and $2$-chains in $C_*(v)$ naturally form subgroups of $\CC_1$ and $\CC_2$ respectively. We now see that on making the appropriate orientation conventions, the  boundary map $\del_2(v):C_2(v)\to C_1(v)$ is the restriction of the boundary map $\del_2:\CC_2\to \CC_1$.

Consider a normal triangle $D(\Delta,v)$ in the tetrahedron $\Delta$ linking the vertex $v$. We can identify this with the face $\Phi(\Delta,v)$ of $\Delta$ opposite to $v$. This is consistent with the previous identification of the normal arc $\alpha(F,v)$ with the edge $e(F,v)$.

Let $a'=a'(\Delta,v)$ be the unit vector normal to $\Phi=\Phi(\Delta,v)$ pointing out of $\Delta$ (see figure~\ref{fig:Delta}). We orient $\Phi$ by declaring a basis $\langle u, w\rangle$ of its tangent space to be positive if and only if the basis $\langle a', u, w\rangle$ is positive. With this orientation, we see that the boundary map on $\CC_2$ restricts to the boundary map on $C_2(v)$.

\begin{proposition}
For the natural inclusions $C_1(v)\hookrightarrow \CC_1$ and $C_2(v)\hookrightarrow \CC_2$, the boundary map $\del:\CC_2\to \CC_1$ restricts to the boundary map $\del_2(v):C_2(v)\to C_1(v)$.
\end{proposition}\label{embed}
\begin{proof}
 It suffices to show that for a normal triangle $D$ linking $v$, the boundary maps coincide. As the boundary in each case is the signed sum of the normal arcs bounded by $D$, it suffices to show that the sign of an arc $\alpha(F,v)$ in the two cases is equal.

In the chain complex $C_2(v)$, the boundary of $D$ is the sum of the edges oriented counterclockwise. This means that if $b'(\Phi,v)$ denotes the vector at the midpoint $m_e(F,v)$ of the edge $e=e(F,v)$ in the plane of $\Phi=\Phi(\Delta,v)$, normal to the edge $e$ and pointing outwards from $\Phi$ (see figure~\ref{fig:Delta}), then the coefficient of the $\alpha(F,v)$ in $\del_2(v)(D)$ is $\varepsilon'=\pm 1$ such that $\langle b'(\Phi,v),\varepsilon'e\rangle$ is positively oriented. By the choice of orientations, this is equivalent to the basis $\langle a'(\Phi,v),b'(\Phi,v),\varepsilon'e\rangle$ being positively oriented in $M$.

Observe that $\langle a'(\Phi,v),b'(\Phi,v),\varepsilon'e\rangle$ is an orthonormal basis that can be obtained by a rotation from $\langle a(F,v),b(F,v),\varepsilon'e\rangle$ (see figure~\ref{fig:Delta}). Hence $\langle a(F,v),b(F,v),\varepsilon'e\rangle$ is a positive basis. By the definition of $\varepsilon(\Delta,F,v)$, it follows that $\varepsilon'=\varepsilon(\Delta,F,v)$. By the definition of $\del:\CC_2\to\CC_1$, it follows that the coefficient of $\alpha(F,v)$ in the boundary of $D$ in the two complexes coincides.

\end{proof}

We see next that the given orientations of the $2$-simplices of $C(v)$ are consistent, in the sense that their sum is a $2$-cycle, and hence the fundamental class in $H_2(S(v),\Z)$.

\begin{proposition}\label{fund}
If $D_j$ are the $2$-simplices in $C(v)$ with the above orientations, then
$$[S(v)]=\sum_j D_j$$
is a $2$-cycle.
\end{proposition}
\begin{proof}
Each edge of $S(v)$, which is a normal arc $\alpha(F,v)$, is the boundary of exactly two $2$-simplices, $D_\pm\subset\Delta_\pm(F)$. Hence it suffices to show that the edge $\alpha(F,v)$ appears with opposite sign in the boundary of $D_\pm$. But we have seen in Lemma~\ref{match} that this is the case when $D_\pm$ are regarded as elements in $\CC_2$. By Proposition~\ref{embed}, the boundary map on $C_2(v)$ is the restriction of the map on $\CC_2$, so the coeffcients of $\alpha(F, v)$ in $\del_2(v) D_\pm$ have opposite signs,  as required.
\end{proof}

\section{Quadrilateral co-ordinates}

We now turn to the question regarding quadrilateral co-ordinates
determining normal surfaces. Quadrilateral co-ordinates are in
bijective correspondence with chains $\zeta\in \CC_2^q$. We shall
henceforth consider such $2$-chains.

Note that admissibility is a condition determined by the quadrilateral
coordinates. We shall assume that $\zeta$ corresponds to non-negative,
admissible quadrilateral coordinates.

Corresponding to the decomposition $\CC_1=\bigoplus_{v\in V}
C_1(v)$, we define homomorphisms $\bdel_v:\CC_2\to C_1(v)$ as the
composition $\pi(v)\circ \del_2$ of the boundary map with the
projection onto $C_1(v)$. As $\CC_1=\bigoplus_{v\in V} C_1(v)$,
for $\xi\in\CC_2$, $\del_2(\xi)=0$ if and only if $\bdel_v(\xi)=0$
for all $v\in V$.

As $\CC_2=\CC_2^t\oplus \CC_2^q$, by Lemma~\ref{match} the 2-chain $\zeta$
corresponds to quadrilateral co-ordinates of a normal surface $F$ with
normal coordinates $\xi$ if and only if there is a $2$-chain
$\zeta'\in\CC^t_2$ with $\del_2 (\zeta+\zeta')=0$. In this case, the
normal co-ordinates of $F$ are $\xi=\zeta+\zeta'$.

We first give a necessary condition for $\zeta$ to correspond to the
quadrilateral coordinates of a normal surface.

\begin{theorem}\label{ideal}
There is a normal surface $F$ with quadrilateral coordinates
corresponding to $\zeta$ if and only if $\bdel_v{\zeta}\in C_1(v)$ is
a boundary in $C_*(v)$ for all $v\in V$.
\end{theorem}
\begin{proof}
First, assume that $\zeta$ corresponds to the quadrilateral
co-ordinates of a surface $F$. Then there is a $2$-chain
$\zeta'\in\CC^t_2$ with $\del
(\zeta+\zeta')=\del \zeta + \del\zeta'=0$. Hence for each vertex $v\in
V$, $\bdel_v \zeta + \bdel_v \zeta'=0$

As $\CC_2^t=\bigoplus C_2(v)$, we can write $\zeta'=\bigoplus_{v\in V}
\zeta'(v)$, $\zeta'(v)\in C_2(v)$. For each $v\in V$,
$\bdel_v \zeta'=\del_2(v) \zeta'(v)$ is a boundary in the complex
$C_*(v)$. Hence $\bdel_v \zeta = -\bdel_v \zeta'$ is also a boundary.

Conversely, if $\bdel_v \zeta$ is a boundary for each $v\in V$, then
there are $2$-chains $\zeta'(v)\in C_2(v)$ with
$\del_2(v) \zeta'(v) =-\bdel_v \zeta$. We claim that we can choose
$\zeta'(v)$ so that all the corresponding (triangle) coordinates are
non-negative. By Proposition~\ref{fund} the sum of the
triangles in $S(v)$ is a cycle $[S(v)]$. By replacing $\zeta'(v)$ by
$\zeta'(v)+k[S(v)]$, for $k$ sufficiently large, we can ensure that
all the co-ordinates are non-negative.

Let $\zeta'=\Sigma_{v\in V}\zeta'(v)\in \CC^t_2$. By construction
$\bdel_v(\zeta+\zeta')=0$ for all $v\in V$, and hence
$\del(\zeta+\zeta')=0$.

Let $\xi=\zeta+\zeta'$. By Lemma~\ref{match}, $\xi$ satisfies the
matching equations. Further, as $\zeta$ is assumed to correspond to
admissible, non-negative quadrilateral coordinates, and the
coordinates of $\zeta'(v)$ are non-negative triangular coordinates, $\xi$ is an admissible,
non-negative solution.
\end{proof}

\begin{remark}
When $\bdel_v{\zeta}$ is a cycle in $C_1(v)$ for all $v \in V$, then $\zeta$ corresponds to the quadrilateral coordinates of a spun-normal surface. The above theorem says that when $\bdel_v \zeta$ is in fact a boundary the spun-normal surface is compact, so that get a normal surface.
\end{remark}

In the important case where $M$ is a manifold, Theorem~\ref{ideal}
takes a particularly useful form.

\begin{corollary}
If $M$ is a manifold, $\zeta$ corresponds to quadrilateral coordinates
of a normal surface if and only if $\bdel_v(\zeta) \in C_1(v)$ is a cycle in $C_*(v)$ for all
$v \in V$.
\end{corollary}
\begin{proof}
This follows from Theorem~\ref{ideal} as $H_1(S(v),\Z)=0$.
\end{proof}

The class $\bdel_v(\zeta)$ is a cycle if and only if its boundary is
zero. This is a condition that is simple to check and also
conceptually very simple.

In the general case, we need to check whether $\bdel_v(\zeta)$ is a
cycle and represents the trivial homology class. The latter can be
checked, for instance, by evaluating on a basis of cohomology.

We now turn to Casson-Rubinstein-Tollefson's observation on uniqueness. The following is a
useful way to state the result.

\begin{theorem}[Casson-Rubinstein-Tollefson]
Let $\zeta$ be an admissible, non-negative set of quadrilateral
coordinates that can be represented by a normal surface. Then there is
a set of admissible, non-negative normal surface coordinates $\xi$
corresponding to $\zeta$ such that if $\xi'$ is another set of such
coordinates, then $\xi'=\xi+\sum_{v\in V}m_v[S(v)]$, with $m_v\geq 0$.
\end{theorem}
\begin{proof}
By Theorem~\ref{ideal}, $\bdel_v \zeta$ is the boundary
of a $2$-chain $\zeta'(v)\in C_2(v)$. If $\zeta''(v)$ is another such
$2$-chain, then $\zeta'(v)-\zeta''(v)$ is a $2$-cycle, hence
represents an element of the homology $H_2(S(v),\Z)$. As
$H_2(S(v),\Z)=\Z$ and is generated by $[S(v)]$,
$\zeta''(v)=\zeta'(v)+m[S(v)]$.

Consider the coefficients of the triangles of $S(v)$ in $\zeta'(v)$ and
let $m$ be the smallest such coefficient. The chain $\zeta'(v)-m[S(v)]$
then has all coefficients non-negative and at least one coefficient
zero. Further, if we replace $\zeta'(v)$ by $\zeta'(v)-m[S(v)]$, we see
that for any non-negative chain $\zeta''(v)$ with $\del(v) \zeta''(v)=\del(v)
\zeta'(v)$, $\zeta''(v)=\zeta'(v)+m'[S(v)]$ with $m'\geq 0$.

Now let $\zeta'=\sum_{v\in V} \zeta'(v)$ and let
$\xi=\zeta+\zeta'$. It is easy to see that $\xi$ is as claimed.
\end{proof}

Let $S$ be a normal surface, and let $(S)$ denote its quadrilateral coordinates. Then the above theorem says that there exists a normal surface $F$ with $(F) = (S)$ such that if $F'$ is any other normal surface with $(F') = (S)$ then $F'$ is the union of $F$ with some vertex-linking surfaces.

\acknowledgements{Tejas Kalelkar acknowledges the SPM Fellowship of the Council of Scientific and Industrial Research for financial support.}

\bibliographystyle{amsplain}

\end{document}